\documentclass[9pt]{extarticle}

\usepackage[letterpaper,top=0.60in,bottom=0.62in,left=0.65in,right=0.65in]{geometry}
\usepackage{lmodern}
\usepackage{authblk}
\usepackage{graphicx}
\usepackage{amsfonts,amssymb}
\usepackage[numbers,sort&compress]{natbib}
\usepackage{microtype}
\usepackage{titlesec}
\usepackage{fix-cm}
\usepackage{mathtools}
\usepackage[T1]{fontenc}
\usepackage{courier}
\usepackage{amsthm}
\usepackage{bm}
\usepackage{xcolor}
\usepackage{tikz}
\usetikzlibrary{arrows.meta,calc,positioning,fit,decorations.pathreplacing}
\usepackage{enumitem}
\usepackage{url}
\usepackage{tabularx}
\usepackage{placeins}
\usepackage{float}
\usepackage{needspace}

\definecolor{knotBlue}{RGB}{0,101,165}
\definecolor{knotOrange}{RGB}{224,96,24}
\definecolor{knotRed}{RGB}{130,0,0}
\definecolor{knotLightBlue}{RGB}{239,246,250}

\newtheorem{theorem}{Theorem}
\newtheorem{corollary}[theorem]{Corollary}
\newtheorem{proposition}[theorem]{Proposition}

\theoremstyle{definition}

\theoremstyle{remark}

\renewcommand{\Pr}{\mathbb{P}}

\newcommand{\neqp}{\mathrm{neq}}

\newcommand{\SILocation}[1]{SI Appendix, #1}
\newcommand{\SIShadowSignLaw}{Lemma~S1.1}
\newcommand{\SIExactCollisionCalculus}{Theorem~S2.2}
\newcommand{\SIUniversalBirthdayLaw}{Theorem~S2.3}

\newcommand{\SIPoissonWindow}{Theorem~S2.5}
\newcommand{\SIHeavyAtomCounterexample}{Proposition~S2.6}

\newcommand{\SIOpenedTrefoilAdmissibility}{Lemma~S4.1}
\newcommand{\SIShadowTransfer}{Lemma~S4.2}
\newcommand{\SIShadowSelection}{Lemma~S4.3}
\newcommand{\SIGenuineInterface}{Lemma~S4.5}
\newcommand{\SIDistanceThreeBuffer}{Corollary~S4.6}
\newcommand{\SISimultaneousSeparators}{Lemma~S4.7}
\newcommand{\SITwoBraidClosure}{Lemma~S4.8}
\newcommand{\SIPersistentFactorization}{Proposition~S4.9}
\newcommand{\SIBinomialAtom}{Lemma~S4.10}
\newcommand{\SIDeterminantAntiConcentration}{Theorem~S4.11}

\newcommand{\SIFallbackIdentities}{Proposition~S5.1}

\newcommand{\SIFixedDeterminant}{Theorem~S5.3}
\newcommand{\SIFixedKnot}{Corollary~S5.4}

\newcommand{\SIDeterminantComplexity}{Proposition~S5.7}
\newcommand{\SINoFallbackDatabase}{Corollary~S5.8}
\newcommand{\SIMachineLearningBenchmark}{Proposition~S5.9}
\newcommand{\SIFiniteMoments}{Theorem~S6.1}

\newcommand{\SIFiniteBirthdayLaw}{Theorem~S6.3}
\newcommand{\SIExactFiniteCurves}{Theorem~S6.4}

\newcommand{\SISupportCompression}{Corollary~S6.6}
\newcommand{\SIFiniteBenchmark}{Proposition~S7.1}
\newcommand{\SICascadeSingletons}{Lemma~S8.1}

\newcommand{\SICascadeWorkload}{Theorem~S8.3}
\newcommand{\SICascadeOrdering}{Proposition~S8.4}
\newcommand{\SICascadeBreakEven}{Proposition~S8.5}

\newcommand{\SIReproducibilitySection}{Section~S9}
\newcommand{\SIDeterminantRefinementTable}{Table~S4}

\usepackage[colorlinks=true,allcolors=knotBlue]{hyperref}

\newcommand{\Endparasplit}{}
\setlist{nosep}
\titlespacing*{\section}{0pt}{1.8ex plus 0.4ex minus 0.2ex}{0.7ex}
\titlespacing*{\subsection}{0pt}{1.4ex plus 0.3ex minus 0.2ex}{0.5ex}
\titlespacing*{\paragraph}{0pt}{1.0ex plus 0.2ex minus 0.1ex}{0.6em}

\title{\bfseries Knots, black holes, databases, and birthdays:\\Collision entropy of knot invariants}
\author[a,d]{Pedro Olivares-S\'anchez}
\author[b,d]{Edison Jessie V\'azquez Gordillo}
\author[c]{Radmila Sazdanovi\'c}
\author[d]{Carlos Alfonso Ruiz Guido}
\author[e]{Aldo Guzm\'an-S\'aenz}
\author[f]{Renato Osvaldo Salmer\'on-Garc\'ia}
\author[b]{Ramiro L\'opez-V\'azquez}
\author[b]{Ernesto Lupercio}

\affil[a]{Department of Mathematics, Instituto Tecnol\'ogico Aut\'onomo de M\'exico (ITAM), Mexico City 01080, Mexico}
\affil[b]{Department of Mathematics, Centro de Investigaci\'on y de Estudios Avanzados del Instituto Polit\'ecnico Nacional (CINVESTAV-IPN), Mexico City 07360, Mexico}
\affil[c]{Department of Mathematics, North Carolina State University, Raleigh, NC 27695, USA}
\affil[d]{Colegio de Matem\'aticas Bourbaki, M\'erida, Yucat\'an, Mexico}
\affil[e]{Independent Researcher}
\affil[f]{School of Engineering and Sciences, Tecnol\'ogico de Monterrey, Campus Santa Fe, Mexico City 01389, Mexico}
\date{Preprint prepared August 23, 2026}

\hypersetup{
  pdftitle={Knots, black holes, databases, and birthdays: Collision entropy of knot invariants},
  pdfauthor={Pedro Olivares-Sánchez et al.},
  pdfsubject={Preprint},
  pdfkeywords={knot invariants, collision entropy, birthday paradox, random knot diagrams, computational black holes}
}

\begin{document}
\raggedbottom
\maketitle

\begin{abstract}
A knot invariant is a fingerprint shared by equivalent knots: unequal values certify inequivalence, while equal values may conceal different knots. Under two rooted random-diagram models, we prove that the incomplete normalized Alexander polynomial separates random pairs yet collides in a growing database. For an $n$-crossing diagram $D_n$, let $K(D_n)$ be its knot. If $D_n$ and $D_n'$ are independent, then $\Pr\{\Delta_{K(D_n)}=\Delta_{K(D_n')}\}=O(n^{-1/2})$, and every fixed normalized Alexander polynomial is exponentially rare. Since $\det K=|\Delta_K(-1)|$, equality of two Alexander polynomials forces determinant equality, while equality with a fixed polynomial fixes the determinant. With exponentially high probability, $D_n$ has linearly many disjoint opened-trefoil slots. Conditional on the shadow and exterior signs, the slot indicators are independent Bernoulli$(1/4)$ variables, and their sum is a binomial coordinate in the $3$-adic determinant valuation. Hence $\sup_{a\geq1}\Pr\{\det K(D_n)=a\}=O(n^{-1/2})$.

For a discrete invariant $I$, let $\alpha_n(I)=\Pr\{I(D_n)=I(D_n')\}$ be the random-pair equality probability. Its collision entropy is $H_2(I(D_n))=-\log\alpha_n(I)$. An independent $M$-sample has $\binom{M}{2}\alpha_n(I)$ colliding pairs on average and birthday scale $\alpha_n(I)^{-1/2}$. If $M_n^2\alpha_n(I)\to\infty$, a repeated value occurs with probability tending to one even when $\alpha_n(I)\to0$. For the determinant, $M=o(n^{1/4})$ suffices for collision freedom with high probability; a matching lower bound is open. Exact finite-population formulas treat fixed censuses with unequal fiber sizes. A cascade sends unresolved records to a later invariant; we calculate its expected cost. In an invariant-first comparison, $\alpha_n(I)$ is the probability of calling a complete procedure. On the balanced pair-classification benchmark, the normalized-Alexander rule has balanced accuracy $1-O(n^{-1/2})$ but cannot distinguish inequivalent pairs in one fiber.
\end{abstract}

\medskip

A knot is a closed curve embedded in three-dimensional space. A knot diagram records a projection of that curve to the plane, together with the over--under information at each crossing. Two diagrams represent the same knot when the corresponding curves can be deformed into one another in three dimensions without cutting a strand or passing one strand through another. The basic comparison problem is to decide whether two given diagrams represent the same knot.

A knot invariant assigns a value to the represented knot. Equivalent knots must have the same value, so different values certify inequivalence. Equal values need not certify equivalence: an incomplete invariant assigns the same value to some distinct knots. The normalized Alexander polynomial is a classical invariant of this kind. We ask whether an incomplete invariant can separate almost every random pair and still repeat somewhere in a sufficiently large database.

We work with two probability laws on diagrams having exactly $n$ crossings. In each model a diagram is rooted by marking one directed half-edge of the underlying planar map. Under either law, the Alexander polynomial separates two independent diagrams with probability $1-O(n^{-1/2})$, and the probability of every fixed normalized Alexander polynomial decays exponentially. These assertions concern the specified sampling laws and do not imply pointwise injectivity. A database of $M$ independently sampled diagrams contains $\binom{M}{2}$ pairs. Since this number grows quadratically, a repeated value can occur with probability tending to one even while a single random pair is separated with probability tending to one. Pairwise separation and database uniqueness therefore have different thresholds.
\Endparasplit

The same distinction applies to any discrete fingerprint. Fix $n$, let $D_n$ be a random object of complexity $n$, and let $I(D_n)$ take values in a finite or countable set.  Here $I$ is an invariant: equivalent objects receive the same value. A fingerprint is either one such map $I$ or a finite tuple $J=(I_1,\ldots,I_k)$ of them. Its level sets are the fingerprint fibers. We suppress the fixed ambient $n$ from atom and sample subscripts, writing
\begin{equation}
 p_v=\Pr\{I(D_n)=v\},
 \qquad
 \alpha_n(I)=\sum_v p_v^2.
 \label{eq:opening-alpha}
\end{equation}
For an independent draw $D_n'$ from the same law, this quadratic mass is exactly the probability of an invariant collision,
\[
 \alpha_n(I)=\Pr\{I(D_n)=I(D_n')\}.
\]
Its negative (natural) logarithm,
\[
 H_2(I(D_n))=-\log\alpha_n(I),
\]
is the order-two R\'enyi entropy, also called the collision entropy \cite{renyi1961}.  When unequal invariant values prove non-equivalence, $1-\alpha_n(I)$ is the probability that the invariant resolves one random comparison immediately.

The equality event also has a computational meaning. An invariant-first algorithm computes $I$, returns \textsc{No} when the two values differ, and invokes a complete equivalence procedure when they agree. The second step is necessary because the invariant alone does not decide pairs in a common fiber. We call this equality region the \emph{invariant black hole}. Generic-case complexity studies algorithms that are efficient on an asymptotically dominant set while exceptional inputs retain the unresolved behavior \cite{kapovichetal2003,gilmanetal2007,hamkinsmiasnikov2006,myasnikovrybalov2008}; later work calls such an exceptional set the ``black hole of the algorithm'' \cite{rybalov2013}. For independent inputs, its probability is $\alpha_n(I)$. The term is algorithm-relative: membership says nothing about intrinsic difficulty.

Now sample an independent database $D_n^{(1)},\ldots,D_n^{(M)}$ of copies of $D_n$ and let
\[
 C_M(I)=\sum_{1\le i<j\le M}
 \mathbf 1_{\{I(D_n^{(i)})=I(D_n^{(j)})\}}
\]
count colliding pairs.  Then
\begin{equation}
 \mathbb E C_M(I)=\binom M2\alpha_n(I).
 \label{eq:opening-workload}
\end{equation}
The random variable $C_M(I)$ is also the number of complete-procedure calls made by the invariant-first algorithm in an all-pairs comparison. Thus $\alpha_n(I)$ is the random-pair equality probability, the probability of entering the invariant black hole, and the factor in the expected database workload $\binom M2\alpha_n(I)$.

Since an $M$-record database has $\binom M2$ pairs, the preceding expectation suggests a threshold when $M^2\alpha_n(I)$ is of constant order. The universal law proved below makes this precise. Its distribution-free birthday scale is
\begin{equation}
 M_{\mathrm B}(n)=\alpha_n(I)^{-1/2}
 =\exp\!\left(\frac{H_2(I(D_n))}{2}\right).
 \label{eq:opening-birthday-scale}
\end{equation}
Databases of size $o(M_{\mathrm B}(n))$ are collision-free with high probability, whereas databases of size $\omega(M_{\mathrm B}(n))$ contain a collision with high probability. Thus, if $\alpha_n(I)\to0$ and $M_n^2\alpha_n(I)\to\infty$, a random pair is separated with probability tending to one while a database of size $M_n$ contains a repeated fingerprint with probability tending to one. This is the birthday phenomenon for an arbitrary discrete law, expressed in terms of collision entropy \cite{diaconismosteller1989,kimmontenegroperestetali2010}.

We now return to knots. For a knot $K$, its determinant is $\det K=|\Delta_K(-1)|$. Equality of normalized Alexander polynomials therefore implies equality of determinants. If the Alexander polynomial equals a fixed polynomial $P$, then the determinant equals the fixed integer $|P(-1)|$. Thus a uniform bound on determinant probabilities controls both equality between two random Alexander polynomials and equality with any fixed Alexander polynomial.

We prove this determinant bound in two rooted random-diagram models \cite{cantarellachapmanmastin2016,chapman2017}. The connected sum of two knots is formed by removing a short arc from each and joining the resulting pairs of endpoints. Its determinant is the product of their determinants; a trefoil has determinant $3$, and the unknot has determinant $1$. The shadow of a diagram is its knot projection with the crossing information removed. An opened-trefoil slot is a small region whose internal crossing signs can produce either a trefoil summand or an unknot summand. With exponentially high probability, the shadow contains linearly many disjoint such slots. After the shadow and all exterior crossing signs have been exposed, the signs inside the slots remain independent. Cutting along a sphere that meets the knot in two points separates a local knot factor. The selected slots admit a single family of pairwise disjoint such spheres for every internal redecoration, and replacing the selected tangles by fixed trivial arcs gives the same exterior knot for every internal sign vector. Each slot contributes a trefoil with conditional probability $1/4$ and an unknot with conditional probability $3/4$. The number of trefoil outcomes is a binomial coordinate in the $3$-adic valuation, the exponent of $3$ in the determinant. Consequently,
\begin{equation}
 \sup_{a\ge1}\Pr\{\det K(D_n)=a\}=O(n^{-1/2}).
 \label{eq:opening-main-theorem}
\end{equation}
It follows from \eqref{eq:opening-main-theorem} that two independent random diagrams have equal determinants with probability $O(n^{-1/2})$, and a determinant-first comparison invokes its complete procedure with vanishing probability. For a fixed target, determinant equality forces the binomial coordinate into a fixed lower tail, so this probability decays exponentially. For independent databases, the same estimate gives
\[
 M_n=o(n^{1/4})
 \quad\Longrightarrow\quad
 \Pr\{\text{any determinant-first fallback call}\}\to0.
\]
This is a sufficient condition obtained from the upper bound $\alpha_n(\det)=O(n^{-1/2})$; no converse at the exponent $1/4$ is proved. The birthday scale is $M_{\mathrm B}(n)=\alpha_n(\det)^{-1/2}$. The bound gives $M_{\mathrm B}(n)=\Omega(n^{1/4})$, but its order remains unknown.

The preceding database model uses independent samples from a probability law. A finite knot census is different: it is a fixed population of distinct entries sampled without replacement. If the census has $T$ entries and invariant-fiber sizes $m_1,\ldots,m_R$, the no-collision probability for a uniform $M$-subset is $e_M(m_1,\ldots,m_R)/\binom TM$, where $e_M$ is the $M$th elementary symmetric polynomial. Its numerator counts collision-free subsets by choosing one entry from each of $M$ distinct fibers. The number of occupied values alone is insufficient: a few large fibers can cause early collisions even when many values occur. An invariant cascade applies a sequence of invariants, sending only records that agree on all earlier values to the next test. Fiber sizes determine the expected numbers of unresolved objects and pairs after each stage, yielding a two-stage ordering criterion and a marginal break-even criterion.

Recent work studies related finite-population statistics. D\l{}otko, Gurnari, and Sazdanovi\'c introduce Mapper-type methods in \cite{dlotko2024mapper}; their subsequent knot-theoretic study attributes to Ernesto Lupercio the observation that unique-value ratios may decrease even while the probability of separating two random knots tends to one \cite{dlotkogurnarisazdanovic2025}. Tubbenhauer and Zhang compare quantum invariants on cumulative prime-knot censuses using value counts and repeated-pair experiments \cite{tubbenhauerzhang2025}. Kelom\"aki et al. prove exponential decay of value and singleton-detection ratios, together with exponential growth of typical fibers, in cumulative prime-alternating-link populations \cite{kelomakietal2025}. Our knot-specific theorem instead concerns two-sample collision probabilities in rooted fixed-crossing ensembles and proves a moving-target determinant bound.

The argument uses established results on coincidence probabilities, R\'enyi entropy, generic-case complexity, knot invariants, and random-diagram patterns \cite{renyi1961,kapovichetal2003,gilmanetal2007,hamkinsmiasnikov2006,myasnikovrybalov2008,rybalov2013,chapman2017}. Its new input is the conditional binomial coordinate in the determinant valuation. This gives the anti-concentration estimate and its consequences for random pairs, databases, and invariant-first computation.

\begin{figure}[t]
\centering
\includegraphics[width=\linewidth]{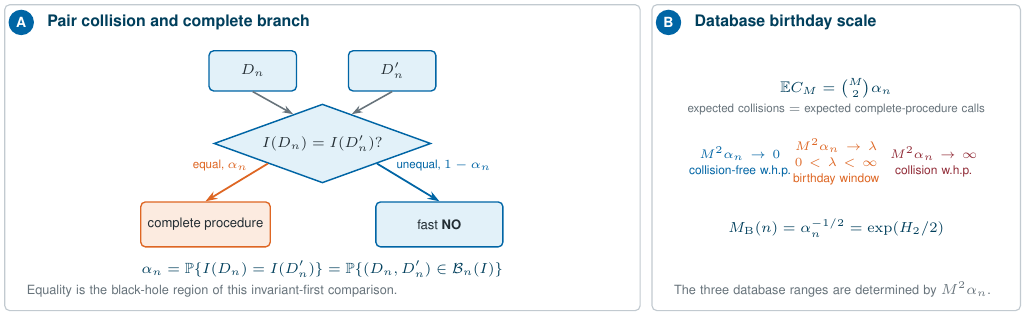}
\caption{Collision probability for a pair and a database. (A) The invariant $I$ assigns a fingerprint value to each object. For two independent objects, $\alpha_n=\mathbb P\{I(D_n)=I(D_n')\}$ is the probability that their values agree. Unequal values certify inequivalence and give a fast \textsc{No}. Equal values leave the comparison unresolved by $I$, so an invariant-first algorithm calls its complete procedure. The term black hole refers only to this algorithm-relative equality region; it does not imply intrinsic hardness. (B) For $M$ independent records, the expected number of colliding pairs, and hence of complete-procedure calls in an all-pairs comparison, is $\mathbb E C_M=\binom{M}{2}\alpha_n$. The database is collision-free with high probability when $M^2\alpha_n\to0$ and contains a collision with high probability when $M^2\alpha_n\to\infty$. The critical window is $M^2\alpha_n\to\lambda\in(0,\infty)$. The birthday scale is $M_{\mathrm B}(n)=\alpha_n^{-1/2}=\exp(H_2/2)$, where $H_2=-\log\alpha_n$ is the collision entropy.}
\label{fig:black-hole-architecture}
\end{figure}

\section{Collision entropy and the inverse birthday law}
\label{sec:birthday-law}
The pair-versus-database transition depends only on the discrete distribution of the fingerprint, not on knot theory. We therefore prove the birthday law for an arbitrary discrete fingerprint distribution.

For each $n$, let $X_n=I(D_n)$ have a discrete law, and fix this ambient $n$ throughout the following definitions.  Write $X_n'{}$ for an independent copy and $X_n^{(1)},\ldots,X_n^{(M)}$ for independent copies of $X_n$, and set
\[
 p_v=\Pr\{X_n=v\},
 \qquad
 \alpha_n=\sum_vp_v^2,
 \qquad
 \beta_n=\sum_vp_v^3,
 \qquad
 p_n^{\max}=\sup_v p_v,
 \qquad
 C_M=\sum_{i<j}\mathbf 1_{\{X_n^{(i)}=X_n^{(j)}\}}.
\]
Thus $\alpha_n$ is the probability that two independent samples agree, and $\beta_n$ is the probability that three independent samples agree. The largest atom satisfies
\[
 (p_n^{\max})^2\le \alpha_n\le p_n^{\max},
\]
so $\alpha_n\to0$ if and only if $p_n^{\max}\to0$. Counting covariances gives
\begin{align*}
 \mathbb E C_M
 &=\binom M2\alpha_n,\\
 \operatorname{Var}(C_M)
 &=\binom M2(\alpha_n-\alpha_n^2)
   +6\binom M3(\beta_n-\alpha_n^2).
\end{align*}
To prove that a collision occurs with high probability, one needs concentration of the collision count. The only pair comparisons that can be dependent are those sharing one sampled object, and their joint probability is $\beta_n$; comparisons on disjoint pairs are independent. Thus $M^2\alpha_n$, the expected collision count up to a constant factor, is the candidate transition parameter. The moment calculation and the scale proof are in \SILocation{\SIExactCollisionCalculus{} and \SIUniversalBirthdayLaw{}}. The next theorem shows that this parameter determines the answer away from the transition scale.

\begin{theorem}[Universal inverse birthday law]\label{thm:inverse-birthday}
For each $n$, let the fingerprint law be any discrete distribution, possibly depending on $n$.  For any integer sequence $M_n\ge2$, let $C_{M_n}$ denote the collision count among $M_n$ iid samples from the law at level $n$.  Then
\[
 M_n^2\alpha_n\longrightarrow0
 \quad\Longrightarrow\quad
 \Pr\{C_{M_n}>0\}\longrightarrow0,
\]
whereas
\[
 M_n^2\alpha_n\longrightarrow\infty
 \quad\Longrightarrow\quad
 \Pr\{C_{M_n}>0\}\longrightarrow1.
\]
Equivalently, with
\[
 M_{\mathrm B}(n):=\alpha_n^{-1/2}
 =\exp\!\left(\frac{H_2(X_n)}{2}\right),
\]
databases of size $o(M_{\mathrm B}(n))$ are collision-free with high probability, while databases of size $\omega(M_{\mathrm B}(n))$ contain a collision with high probability.
\end{theorem}

\begin{proof}
The first implication follows from Markov's inequality and
$\mathbb E C_{M_n}=\binom{M_n}{2}\alpha_n$.
For the second, the power-sum inequalities
\[
 \alpha_n^2\le\beta_n\le\alpha_n^{3/2}
\]
imply
\[
 \frac{\operatorname{Var}(C_M)}
 {(\mathbb E C_M)^2}
 \le
 \frac{4}{M^2\alpha_n}
 +\frac{4}{M\sqrt{\alpha_n}}.
\]
If $M_n^2\alpha_n\to\infty$, both terms vanish. Chebyshev's inequality then gives
$\Pr\{C_{M_n}=0\}\to0$.
\end{proof}

The threshold occurs when $M_n$ is comparable to $\alpha_n^{-1/2}$. At that scale the expected collision count remains bounded, so the preceding zero-one law does not determine the limiting probability. If $M_n^2\alpha_n\to\lambda\in(0,\infty)$ and, in addition,
\[
 M_n^3\beta_n\longrightarrow0,
\]
then $M_n\to\infty$, $\mathbb E C_{M_n}\to\lambda/2$, and
\[
 C_{M_n}\Longrightarrow\operatorname{Poisson}(\lambda/2).
\]
The proof is in \SILocation{\SIPoissonWindow{}}. The construction in \SILocation{\SIHeavyAtomCounterexample{}} has a moving heavy atom and fails to have this Poisson limit, so the additional condition cannot be omitted in general. Such failures are familiar in Poisson approximation with local dependence or rare clusters \cite{arratiagoldsteingordon1989,barbourholstjanson1992}.

\begin{corollary}[Inverse birthday phenomenon]\label{cor:inverse-birthday}
Suppose $\alpha_n\to0$ and $M_n^2\alpha_n\to\infty$. Then
\[
 \Pr\{X_n\ne X_n'\}\longrightarrow1,
 \qquad
 \Pr\{C_{M_n}>0\}\longrightarrow1.
\]
Thus one random pair is asymptotically separated while a database of size $M_n$ asymptotically contains a repeated fingerprint.
\end{corollary}

For knot invariants, it remains to bound $p_n^{\max}$ uniformly. The number of occupied invariant values alone gives no such bound, since a few fibers may account for most of $\alpha_n$. A bound for each fixed value would also be insufficient: the most likely value may change with $n$. The following section controls the supremum over all determinant values, including values that move with $n$, using trefoil-slot contributions that are independent after conditioning on the shadow and the exterior crossing signs.

\section{Trefoil slots and determinant anti-concentration}
\label{sec:trefoil-slots}
A knot invariant has the same value on every diagram of a given knot, although different knots may share that value.  For a random diagram, anti-concentration means that no single value of the invariant has substantial probability.  We prove such an estimate for the knot determinant,
\[
 \det K:=\lvert\Delta_K(-1)\rvert,
\]
the absolute value of the Alexander polynomial evaluated at $-1$, in two random-diagram ensembles introduced by Chapman \cite{chapman2017}.  For a positive integer $m$, the $3$-adic valuation $v_3(m)$ is the exponent of $3$ in its prime factorization.  With exponentially high probability, the unsigned planar diagram contains linearly many separated places where three unexposed crossing choices can insert either a trefoil or an unknot.  These choices produce a binomial summand in $v_3(\det K)$ and hence prevent the determinant from concentrating at one integer.

\subsection{Random signs on a fixed shadow}
Let $\mathcal D_n^{\mathrm{gen}}$ be Chapman's rooted general one-component diagrams with $n$ crossings and let $\mathcal D_n^{\mathrm{red}}$ be the corresponding rooted reduced class.  A root is a distinguished dart of the planar map, and this mark is retained in the counting; one-component means that the diagram represents a knot rather than a link.  Write $\mathcal S_n^{\star}$ for the shadow class obtained from $\mathcal D_n^{\star}$ by forgetting which strand passes over at each crossing, where $\star\in\{\mathrm{gen},\mathrm{red}\}$.  Thus a shadow is a rooted four-valent planar map with the strand continuation retained.  The reduced class excludes shadows with a disconnecting crossing vertex.  One-componentness and reducedness are properties of the shadow.  Restoring one of the two over--under choices at each vertex gives a decorated lift.  Every rooted shadow in either class therefore has exactly $2^n$ decorated lifts, and a uniform diagram $D_n$ can be sampled exactly by
\[
 \text{uniform rooted shadow}
 \quad+\quad
 \text{independent fair signs at its }n\text{ crossings}.
\]
The slots are selected from the shadow before any internal signs are exposed. The product law is proved in \SILocation{\SIShadowSignLaw{}}.

Start with the standard three-crossing trefoil diagram and cut one edge away from its crossings.  The result is a one-strand tangle with two boundary ends; $Q$ denotes its shadow.  Rejoining the ends gives the doubled-triangle graph.  The tangle is prime, two-leg-prime, and reduced and admits an asymmetric fixed decoration $P$, as required by Chapman's theorem.

Chapman's connect-sum construction and pattern theorem imply that $P$ occurs linearly many times as a connected-sum tangle, outside an exponentially small exceptional set \cite[Propositions 3 and 10]{chapman2017}.  Forgetting the signs of $P$ leaves the same abundance of unsigned $Q$-slots.  A deterministic rule based only on the rooted planar map then selects a separated subfamily: after occurrences within bounded graph distance are removed, it retains
\[
 s_n=\lfloor\delta n\rfloor
\]
$Q$-slots whose three-vertex carriers are at pairwise graph distance at least three, on an event $\mathcal G_n$ satisfying
\[
 \Pr(\mathcal G_n^c)\le C_0e^{-c_0n}.
\]
The fixed decoration $P$ is thus only a witness for Chapman's theorem. Its signs are discarded before the random diagram is decorated. Admissibility, passage to shadows, and shadow-measurable selection are proved in \SILocation{\SIOpenedTrefoilAdmissibility{}, \SIShadowTransfer{}, and \SIShadowSelection{}}.

\begin{figure}[t]
\centering
\includegraphics[width=\linewidth]{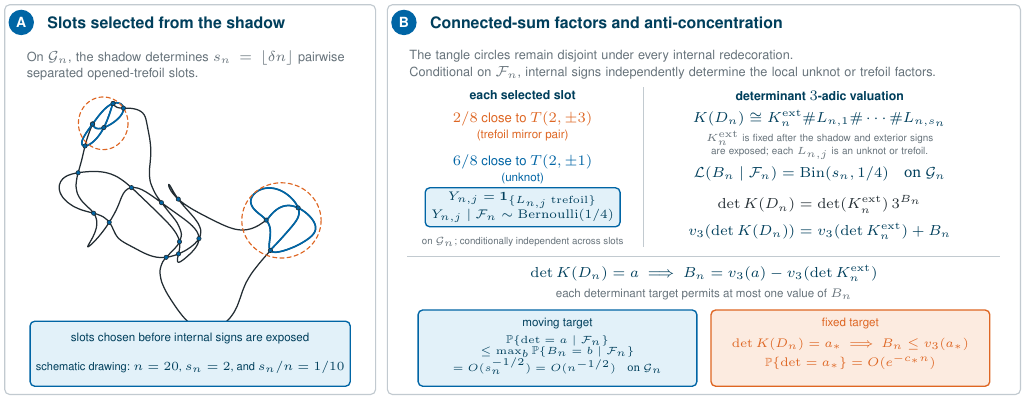}
\caption{Opened-trefoil slots and the binomial coordinate. A shadow is the four-valent planar map obtained from a knot diagram by forgetting the over--under choice at each crossing. (A) On the event $\mathcal G_n$, a rule depending only on the shadow selects $s_n=\lfloor\delta n\rfloor$ pairwise separated three-crossing slots before their signs are exposed. The schematic has $n=20$ and $s_n=2$, so its displayed ratio $1/10$ does not specify the theorem's constant $\delta>0$. (B) The displayed tangle circles correspond to pairwise disjoint separating spheres that work for every later choice of the internal signs. After the shadow and all exterior signs have been revealed, $\mathcal F_n$ records this information. Two of the eight equally likely sign triples close to the trefoil mirror pair and six close to the unknot. The slot indicators are therefore conditionally independent Bernoulli$(1/4)$ variables, and their sum $B_n$, the number of trefoil factors, satisfies $\mathcal L(B_n\mid\mathcal F_n)=\operatorname{Bin}(s_n,1/4)$. Replacing every local tangle by a trivial arc gives the exterior knot $K_n^{\mathrm{ext}}$. Cutting along the separating spheres gives $K(D_n)\cong K_n^{\mathrm{ext}}\#L_{n,1}\#\cdots\#L_{n,s_n}$, where $\#$ denotes connected sum. The determinant $\det K=\lvert\Delta_K(-1)\rvert$ is multiplicative under connected sum; it equals $3$ for a trefoil and $1$ for the unknot. Hence $v_3(\det K(D_n))=v_3(\det K_n^{\mathrm{ext}})+B_n$, where $v_3(m)$ is the exponent of $3$ in $m$. Equality with a determinant target $a$ forces $B_n=v_3(a)-v_3(\det K_n^{\mathrm{ext}})$, whose largest conditional atom is $O(n^{-1/2})$. For fixed $a_*$, determinant equality implies $B_n\le v_3(a_*)$, and this lower-tail probability is exponentially small.}
\label{fig:hidden-binomial}
\end{figure}

\subsection{The Bernoulli contribution of one slot}
Work on $\mathcal G_n$, where the selected family has exactly $s_n$ slots.  Reveal the complete shadow and every crossing sign outside their three-vertex carriers, but leave the three signs in each slot unexposed.  Let $\mathcal F_n$ be the sigma-field generated by the revealed data.  The event $\mathcal G_n$, the selected list, and the separator geometry are $\mathcal F_n$-measurable.  Conditional on $\mathcal F_n$, the $3s_n$ internal signs remain independent and fair.

For the $j$th slot, write the reference-relative signs as $\boldsymbol\varepsilon_{n,j}=(\varepsilon_{n,j,1},\varepsilon_{n,j,2},\varepsilon_{n,j,3})$.  There are eight equally likely triples.  The two unanimous triples close to the two mirror images of the trefoil; each of the other six closes to the unknot.  Let $T(2,r)$ denote the closure of the two-strand braid of exponent $r$.  After a two-braid generator is chosen consistently with the embedding, the local closure is
\[
 T\!\left(2,\eta_j(\varepsilon_{n,j,1}+\varepsilon_{n,j,2}+\varepsilon_{n,j,3})\right),
 \qquad \eta_j\in\{-1,+1\},
\]
where $\eta_j$ is fixed by the shadow.  Hence
\[
 Y_{n,j}:=\mathbf 1\{\varepsilon_{n,j,1}=\varepsilon_{n,j,2}=\varepsilon_{n,j,3}\},
 \qquad
 \mathcal L\bigl((Y_{n,1},\ldots,Y_{n,s_n})\mid\mathcal F_n\bigr)
 =\operatorname{Bernoulli}(1/4)^{\otimes s_n}
\]
almost surely on $\mathcal G_n$.

This local calculation does not by itself give a simultaneous connected-sum decomposition.  For that, the cutting spheres must be pairwise disjoint and must remain fixed when any collection of internal signs is changed.  Each selected occurrence contains its three crossing vertices and all five edges between them; exactly two distinct joining edges connect it to the rest of the shadow. Distinct internal carriers lie at graph distance at least three. These facts produce a single shadow-measurable family of pairwise disjoint balls that separates all selected occurrences for every choice of their internal crossing signs.

On an $\mathcal F_n$-fiber, let $D_n[\mathbf e]$ denote the diagram obtained by assigning an arbitrary internal sign vector $\mathbf e=(\mathbf e_1,\ldots,\mathbf e_{s_n})$.  Replace the strand inside each separating ball by a fixed trivial arc.  The resulting knot is the exterior knot $K_n^{\mathrm{ext}}$; it is $\mathcal F_n$-measurable and unchanged by $\mathbf e$.  Cutting along the separating spheres and capping both sides expresses the original knot as a connected sum, denoted by $\#$.  Simultaneously for all $2^{3s_n}$ choices,
\begin{equation}
 K(D_n[\mathbf e])
 \cong K_n^{\mathrm{ext}}\#L_{n,1}(\mathbf e_1)\#\cdots\#L_{n,s_n}(\mathbf e_{s_n}),
 \label{eq:simultaneous-factorization}
\end{equation}
where each local factor is an unknot or a trefoil according to the preceding rule.  The two-edge boundary, distance-three separation, simultaneous separator balls, two-braid closure, and factorization in every $\mathcal F_n$-fiber are proved in \SILocation{\SIGenuineInterface{}, \SIDistanceThreeBuffer{}, \SISimultaneousSeparators{}, \SITwoBraidClosure{}, and \SIPersistentFactorization{}}.

The determinant is multiplicative under connected sum.  The unknot has determinant $1$, while either trefoil has determinant $3$.  For the random internal vector, put $B_n=\sum_{j=1}^{s_n}Y_{n,j}$, the number of trefoil factors.  It follows, almost surely on $\mathcal G_n$, that
\begin{equation}
 \det K(D_n)=\det(K_n^{\mathrm{ext}})3^{B_n},
 \qquad
 v_3\!\left(\det K(D_n)\right)=v_3(\det K_n^{\mathrm{ext}})+B_n,
 \label{eq:hidden-binomial}
\end{equation}
with
\[
 \mathcal L(B_n\mid\mathcal F_n)=\operatorname{Bin}(s_n,1/4).
\]
The same $K_n^{\mathrm{ext}}$ works for every internal sign vector. Conditional on $\mathcal F_n$, $v_3(\det K_n^{\mathrm{ext}})$ is fixed, and the only varying term in the valuation identity is $B_n$, with the displayed binomial law.

\begin{theorem}[Determinant anti-concentration]
\label{thm:det-anti}
There are constants $\delta,c_0,C_0>0$ and $n_0\ge4$ such that, for $n\ge n_0$ and a uniform rooted diagram in either of the two rooted Chapman ensembles,
\[
 \sup_{a\ge1}\Pr\{\det K(D_n)=a\}
 \le
 C_0e^{-c_0n}
 +\sqrt{\frac{2\pi}{3\lfloor\delta n\rfloor}}
 =O(n^{-1/2}).
\]
The estimate is uniform over determinant targets $a=a_n$ that may vary with $n$.
\end{theorem}

\begin{proof}
Fix a target $a$.  On $\mathcal G_n$, condition on $\mathcal F_n$.  \eqref{eq:hidden-binomial} forces $B_n$ to equal the single $\mathcal F_n$-measurable integer
\[
 v_3(a)-v_3(\det K_n^{\mathrm{ext}}),
\]
when that integer is admissible.  If $B\sim\operatorname{Bin}(m,1/4)$, Fourier inversion gives
\[
 \sup_k\Pr\{B=k\}\le\sqrt{\frac{2\pi}{3m}}.
\]
Thus the conditional target probability is bounded by this quantity on $\mathcal G_n$.  Taking expectations and adding $\Pr(\mathcal G_n^c)$ proves the claim.
\end{proof}

If $p_a=\Pr\{\det K(D_n)=a\}$, then
\[
 \alpha_n(\det)=\sum_a p_a^2
 \le \sup_a p_a
 =O(n^{-1/2}),
\]
and hence
\[
 H_2(\det K(D_n))\ge\frac12\log n-O(1).
\]
In particular, two independent diagrams have different determinants with probability $1-O(n^{-1/2})$. Section~\ref{sec:black-holes} applies this estimate to black-hole probabilities and databases of independent diagrams.

\section{Vanishing computational black holes}
\label{sec:black-holes}
The determinant bound has an algorithmic interpretation. Let $I$ be a sound knot invariant, so $I(D)\ne I(D')$ certifies inequivalence, and let $\mathsf{Decide}$ be a complete knot-equivalence procedure. They give the invariant-first algorithm
\begin{equation}
 \mathsf A_I(D,D')=
 \begin{cases}
  \textsc{No},& I(D)\ne I(D'),\\
  \mathsf{Decide}(D,D'),& I(D)=I(D').
 \end{cases}
 \label{eq:invariant-first-wrapper}
\end{equation}
For each complexity $n$, let $\Omega_n$ denote the space of diagram inputs and let $\mu_n$ be the probability law from which a diagram is drawn.  Define the \emph{invariant black hole}
\begin{equation}
 \mathcal B_n(I):=
 \{(D,D')\in\Omega_n^2:I(D)=I(D')\}.
 \label{eq:black-hole-definition}
\end{equation}
Generic-case complexity was introduced for decision problems by Kapovich et al. and developed in subsequent work on decidable and undecidable problems, including the halting problem \cite{kapovichetal2003,gilmanetal2007,hamkinsmiasnikov2006,myasnikovrybalov2008}.  Later work uses the phrase ``black hole of the algorithm'' for the exceptional input set \cite{rybalov2013}.  Here $\mathcal B_n(I)$ means only the set of pairs on which $\mathsf A_I$ calls $\mathsf{Decide}$.  It may contain equivalent knots, inequivalent knots, easy instances, and difficult instances; membership says nothing about intrinsic hardness.

\begin{proposition}[Black-hole and workload identities]
\label{prop:black-hole-identity}
Let $D_n,D_n'$ be independent with law $\mu_n$.  Then
\begin{equation}
 \Pr\{(D_n,D_n')\in\mathcal B_n(I)\}
 =\Pr\{I(D_n)=I(D_n')\}
 =\alpha_n(I).
 \label{eq:black-hole-mass}
\end{equation}
For an independent database $D_n^{(1)},\ldots,D_n^{(M)}$ of copies of $D_n$, let $N_v$ be the observed number of objects with invariant value $v$.  The number $F_M(I)$ of all-pairs fallback calls made by~\eqref{eq:invariant-first-wrapper} is exactly
\begin{equation}
 F_M(I)
 =\sum_v\binom{N_v}{2}
 =C_M(I),
 \qquad
 \mathbb E F_M(I)=\binom M2\alpha_n(I).
 \label{eq:fallback-workload}
\end{equation}
\end{proposition}

\begin{proof}
The event in~\eqref{eq:black-hole-mass} is the equality event defining the collision mass.  Each bucket of size $N_v$ contributes one fallback call for each of its $\binom{N_v}{2}$ unordered pairs, which proves~\eqref{eq:fallback-workload}; expectation then uses the pair-collision identity from Section~\ref{sec:birthday-law}.
\end{proof}

The identity is also proved in \SILocation{\SIFallbackIdentities{}}.  Thus $\alpha_n(I)$ is the probability that one random comparison calls $\mathsf{Decide}$ and the expected fraction of database pairs for which this call is made; its inverse square root is the birthday scale.  Figure~\ref{fig:black-hole-architecture} displays these identities.

\subsection{Random pairs and fixed targets}
Theorem~\ref{thm:det-anti} gives a vanishing black hole for determinant-first comparison.

\begin{corollary}[Vanishing random-pair black hole]
\label{cor:random-pair-black-hole}
Let $D_n,D_n'$ be independent uniform rooted diagrams in either of the two rooted Chapman ensembles considered here.  Then
\begin{equation}
 \Pr\{(D_n,D_n')\in\mathcal B_n(\det)\}
 =\alpha_n(\det)
 =O(n^{-1/2}).
 \label{eq:random-pair-black-hole-rate}
\end{equation}
Consequently, $\mathsf A_{\det}$ returns a sound negative answer on a $1-O(n^{-1/2})$ fraction of random pairs before invoking the complete procedure.
\end{corollary}

There are two distinct comparison problems.  In a random-pair comparison, both determinants vary with the sample.  In a fixed-target comparison, one knot $K_*$ is prescribed and only the random diagram varies.  Its determinant must then equal the fixed integer $\det K_*$.  In the trefoil-slot construction, this can happen only when the number of trefoil contributions stays bounded, although its mean grows linearly with $n$.  This yields an exponential estimate rather than the polynomial random-pair estimate.

\begin{corollary}[Exponential fixed-target black hole]
\label{cor:fixed-target-black-hole}
For every fixed knot $K_*$, there are constants $C_{K_*},c_{K_*}>0$ such that
\begin{equation}
 \Pr\{\det K(D_n)=\det K_*\}
 \le C_{K_*}e^{-c_{K_*}n}.
 \label{eq:fixed-target-exponential}
\end{equation}
Thus determinant-first comparison with $K_*$ invokes its complete fallback with exponentially small probability.
\end{corollary}

The proof is given in \SILocation{\SIFixedDeterminant{} and \SIFixedKnot{}}.  If $a=\det K_*$, then
\[
 \det(K_n^{\mathrm{ext}})3^{B_n}=a
 \quad\Longrightarrow\quad
 B_n\le v_3(a).
\]
On $\mathcal G_n$, the conditional law $B_n\mid\mathcal F_n\sim\operatorname{Bin}(\lfloor\delta n\rfloor,1/4)$ has mean proportional to $n$, whereas $v_3(a)$ is fixed.  A binomial lower-tail estimate, together with the exponentially small probability of $\mathcal G_n^c$, gives the fixed-target bound.  The maximal-atom estimate in Theorem~\ref{thm:det-anti} gives the random-pair bound:
\[
 \text{random pair: }O(n^{-1/2}),
 \qquad
 \text{fixed target: }O(e^{-cn}).
\]

\subsection{Refining invariants}
The determinant is a coarse fingerprint.  A richer invariant refines it when the richer value determines the determinant; every fiber of the richer invariant then lies inside a determinant fiber.  Accordingly, call an invariant $J$ \emph{determinant-refining} if there is a deterministic map $g$ with
\[
 \det K=g(J(K))
\]
for every knot in its stated normalization.  Then every $J$-fiber lies inside a determinant fiber, so
\begin{equation}
 \alpha_n(J)\le\alpha_n(\det),
 \qquad
 \sup_w\Pr\{J(D_n)=w\}
 \le\sup_a\Pr\{\det K(D_n)=a\}.
 \label{eq:refinement-monotonicity}
\end{equation}
Hence every determinant-refining fingerprint inherits the $O(n^{-1/2})$ random-pair collision bound, and every fixed normalized value is exponentially rare.  In particular, $\det K=|\Delta_K(-1)|$ in the usual normalization, so the Alexander polynomial determines the determinant and inherits both bounds.  \SIDeterminantRefinementTable{} in the SI Appendix records the required normalization conventions for the Alexander, Jones, and HOMFLY--PT polynomials, full bigraded Khovanov homology, and full bigraded knot Floer homology.

\eqref{eq:refinement-monotonicity} compares collision probabilities, not evaluation times.  The algorithmic conclusion in Corollary~\ref{cor:random-pair-black-hole} therefore uses the determinant itself.  It can be computed in polynomial bit complexity from a Goeritz matrix by fraction-free integer elimination \cite{gordonlitherland1978,bareiss1968}; the proof is in \SILocation{\SIDeterminantComplexity{}}.  When the determinants differ, this polynomial-time branch gives a conclusive negative answer on a fraction $1-O(n^{-1/2})$ of random pairs.

When the determinants agree, the algorithm calls a complete knot-equivalence procedure.  Such a procedure exists by normal-surface methods \cite{haken1961,hemion1992}, and the complexity of knot and link decision problems is discussed in \cite{hasslagariaspippenger1999,lackenby2021}.  The theorem controls how often this complete branch is called.  The expected and worst-case running times of the full algorithm still depend on that unrestricted procedure.

\subsection{Machine learning outside and inside the black hole}
\label{subsec:machine-learning}
A classifier can attain high accuracy by reproducing a cheap, incomplete invariant while failing to distinguish knots inside its collision fibers.  This possibility is relevant to machine-learning studies of invariant prediction, conjecture discovery from learned correlations, geometric classification, and embeddings designed to be invariant under changes of diagram \cite{daviesetal2021,cravenetal2023,sleimanetal2024,halversonruehle2025}.  Halverson and Ruehle found strong correlations between their learned embeddings and Goeritz-matrix data across the architectures they tested \cite{halversonruehle2025}.  Since the determinant is computed from a reduced Goeritz matrix, Theorem~\ref{thm:det-anti} and the calculation below show that a determinant-only rule already has balanced accuracy $1-O(n^{-1/2})$ under the stated random-pair law.

Let
\[
 \mathcal K_n=\{K(D_n)\cong K(D_n')\},
 \qquad q_n=\Pr(\mathcal K_n)<1.
\]
Consider a balanced binary benchmark whose positive class consists of equivalent pairs and whose negative class is the independent pair law conditioned on $\mathcal K_n^c$.  The invariant-only rule
\[
 h_I(D,D')=\mathbf 1_{\{I(D)=I(D')\}}
\]
uses the convention that $1$ means ``equivalent.''  Its true-positive rate is one because equivalent knots have equal invariant values.  Balanced accuracy is the mean of the true-positive and true-negative rates.  Here it is
\begin{equation}
 \operatorname{BAcc}(h_I)
 =1-\frac12(\frac{\alpha_n(I)-q_n}{1-q_n})
 \ge
 1-\frac{\alpha_n(I)}{2}.
 \label{eq:ml-balanced-accuracy}
\end{equation}
The identity follows from the inclusion $\mathcal K_n\subseteq\mathcal B_n(I)$ and is proved in \SILocation{\SIMachineLearningBenchmark{}}.  For the normalized Alexander polynomial, $\det K=|\Delta_K(-1)|$, so Alexander equality implies determinant equality and $\alpha_n(\Delta)\le\alpha_n(\det)=O(n^{-1/2})$.  Thus this classically incomplete rule has balanced accuracy $1-O(n^{-1/2})$.  For a fixed target $K_*$, the corresponding balanced benchmark (positives representing $K_*$ and negatives drawn from $D_n$ conditioned on $K(D_n)\not\cong K_*$) has invariant-only balanced accuracy $1-O(e^{-cn})$.

Unequal invariant values give sound negative answers, but equality leaves $\mathcal B_n(I)$ unresolved.  Aggregate accuracy can therefore be large even when a classifier uses only information captured by polynomial invariants, an instance of shortcut learning \cite{geirhosetal2020}.  Accuracy conditional on $I(D)=I(D')$ measures what the classifier does within a collision fiber.  Relevant tests use inequivalent pairs conditioned to have the same chosen invariant, including mutant and composite-knot examples \cite{sleimanetal2024}.

\subsection{Determinant comparison in independent databases}
An independent database of $M_n$ diagrams contains $\binom{M_n}{2}$ pairs.  Each pair reaches the complete procedure with probability $O(n^{-1/2})$, so the expected number of fallback calls is $O(M_n^2n^{-1/2})$.  This expectation tends to zero when $M_n=o(n^{1/4})$; Markov's inequality then shows that, with high probability, no pair reaches the complete procedure.  The workload identity makes this argument precise.

\begin{corollary}[No fallback below the proved scale]
\label{cor:no-fallback-database}
Let $D_n^{(1)},\ldots,D_n^{(M_n)}$ be independent uniform rooted diagrams from the complexity-$n$ law in either of the two rooted Chapman ensembles considered here.  For determinant-first all-pairs comparison,
\begin{equation}
 \mathbb E F_{M_n}(\det)
 =O(M_n^2n^{-1/2}),
 \label{eq:expected-fallback-count}
\end{equation}
and
\begin{equation}
 M_n=o(n^{1/4})
 \quad\Longrightarrow\quad
 \Pr\{F_{M_n}(\det)>0\}\longrightarrow0.
 \label{eq:no-fallback-scale}
\end{equation}
Thus, with high probability, the determinant certifies every pair in such a database as inequivalent, and the complete decider is never called.
\end{corollary}

\begin{proof}
Equations~(\ref{eq:fallback-workload}) and (\ref{eq:random-pair-black-hole-rate}) give~\eqref{eq:expected-fallback-count}.  Markov's inequality yields
\[
 \Pr\{F_{M_n}(\det)>0\}
 \le \mathbb E F_{M_n}(\det)
 =O(M_n^2n^{-1/2}),
\]
which tends to zero under the stated hypothesis.
\end{proof}

The full corollary is in \SILocation{\SINoFallbackDatabase{}}.  The upper bound $\alpha_n(\det)=O(n^{-1/2})$ proves the no-fallback conclusion for $M_n=o(n^{1/4})$. The birthday scale is
\[
 M_{\mathrm B}(n)=\alpha_n(\det)^{-1/2},
\]
and the present estimate gives $M_{\mathrm B}(n)=\Omega(n^{1/4})$. No matching lower bound on $\alpha_n(\det)$ is known, so the order of $M_{\mathrm B}(n)$ remains undetermined.

\section{Birthdays in finite knot databases}
\label{sec:finite-databases}
Many knot databases are fixed catalogs rather than independent samples from a random-diagram model.  A census consists of a fixed list of distinct entries, and a random database of size $M$ is then a uniformly chosen $M$-entry subset, sampled without replacement.  This differs from Sections~\ref{sec:trefoil-slots} and \ref{sec:black-holes}, where rooted diagrams are sampled independently from Chapman's ensembles.  All formulas in this section use the finite-population law.

Let a census of $T$ distinct entries be partitioned into nonempty invariant fibers of sizes $m_1,\ldots,m_R$.  For two known-distinct entries, the collision probability is
\begin{equation}
 \alpha^{\neqp}
 =\frac{\sum_v m_v(m_v-1)}{T(T-1)}.
 \label{eq:finite-replacement-correction}
\end{equation}
Choose a uniform $M$-subset of the census.  Let $C_M^{\mathrm{fp}}$ count its equal-invariant pairs, where $\mathrm{fp}$ denotes finite population.  Write $e_M(m_1,\ldots,m_R)$ for the sum, over all choices of $M$ distinct fibers, of the product of their sizes.  This is the $M$th elementary symmetric polynomial.

\begin{proposition}[Finite-population birthday formulas]
\label{prop:exact-finite-census}
For $0\le M\le T$,
\begin{equation}
 \mathbb E C_M^{\mathrm{fp}}
 =\binom M2\alpha^{\neqp},
 \qquad
 \Pr\{C_M^{\mathrm{fp}}=0\}
 =\frac{e_M(m_1,\ldots,m_R)}{\binom TM}.
 \label{eq:exact-finite-birthday}
\end{equation}
Consequently the fiber sizes determine the no-collision probability for every $M$.
\end{proposition}

\begin{proof}
Every sampled pair is a uniform pair of distinct census entries, which gives the expectation.  A collision-free subset chooses one entry from each of $M$ distinct fibers.  Summing the products of the chosen fiber sizes gives $e_M(m_1,\ldots,m_R)$.
\end{proof}

Further moment formulas and the without-replacement zero--one law are proved in \SILocation{\SIFiniteMoments{}, \SIExactFiniteCurves{}, and \SIFiniteBirthdayLaw{}}.  If only the census size $T$ and the number $R$ of occupied invariant values are known, they imply the lower bound
\begin{equation}
 \alpha^{\neqp}\ge\frac{T-R}{R(T-1)}.
 \label{eq:support-compression-bound}
\end{equation}
At fixed $T$ and $R$, balanced integer fibers minimize collision mass and maximize the no-collision probability for every $M$.  In the synthetic $1000$-entry example in \SILocation{\SIFiniteBenchmark{}}, the collision probability first reaches one half at $M_{50}=17$ for a fingerprint with $900$ occupied values and at $M_{50}=38$ for a balanced fingerprint with $500$ occupied values.  Thus $R$ alone does not determine the collision scale; it depends on the full fiber profile.

\section{Adaptive invariant cascades}
\label{sec:cascades}
An invariant cascade evaluates a first invariant on every sampled entry and groups entries with the same value into buckets.  An entry alone in its sampled bucket forms no unresolved pair and does not proceed.  Only entries in buckets of size at least two pass to the next invariant.  Fix a $T$-entry population and sample a uniform $M$-entry subset.  After stage $j$, let $J_j=(I_1,\ldots,I_j)$ be the joint prefix fingerprint.  Two entries lie in the same $J_j$-bucket precisely when they agree on all of $I_1,\ldots,I_j$, and $m_{J_j,v}$ denotes the corresponding population fiber size.

For any prefix $J$, let $S_J(M)$ be the expected number of entries that are alone in their sampled $J$-bucket, let $U_J(M)$ be the expected number that remain in nonsingleton buckets and pass to the next stage, and let $P_J(M)$ be the expected number of sampled pairs with equal $J$-values.  Then
\begin{align}
 S_J(M)
 &=\frac{1}{\binom TM}\sum_v m_{J,v}\binom{T-m_{J,v}}{M-1},
 \label{eq:expected-singletons}\\
 U_J(M)&=M-S_J(M),
 \label{eq:expected-survivors}\\
 P_J(M)&=\binom M2\alpha_J^{\neqp}.
 \label{eq:expected-unresolved-pairs}
\end{align}
These quantities use the joint prefix fibers and require no independence among invariant coordinates.

\begin{proposition}[Expected workload of an adaptive cascade]
\label{prop:adaptive-workload}
Let $1\le M\le T$.  Suppose $I_j$ costs $c_j$ per evaluated entry and the terminal complete procedure costs $f$ per unresolved pair.  If $I_1$ is evaluated on every sampled entry and each later $I_j$ only on entries in nonsingleton sampled $J_{j-1}$-buckets, then
\begin{equation}
 \mathbb E W
 =Mc_1+\sum_{j=2}^{k}c_jU_{J_{j-1}}(M)
   +fP_{J_k}(M).
 \label{eq:adaptive-workload}
\end{equation}
\end{proposition}

The formula follows by linearity of expectation; details are in \SILocation{\SICascadeWorkload{}, \SICascadeOrdering{}, and \SICascadeBreakEven{}}.  For a fixed pair of invariants $A$ and $B$, either order ends with the same joint fingerprint $(A,B)$ and therefore the same terminal unresolved pairs.  The orders differ only in how many entries require the second evaluation.  Consequently,
\begin{equation}
 A\rightarrow B\ \text{is no more expensive than}\ B\rightarrow A
 \quad\Longleftrightarrow\quad
 c_A S_B(M)\le c_B S_A(M),
 \label{eq:two-stage-order}
\end{equation}
whereas the decision to append another invariant depends on its evaluation cost and the reduction in terminal pairs.  Write $P_{J,I}(M)$ for the expected unresolved-pair count associated with the joint fingerprint $(J,I)$.  Adding $I$ after $J$ costs $c_IU_J(M)$ in expected evaluations and saves $f[P_J(M)-P_{J,I}(M)]$ in expected terminal calls.  It does not increase the expected total cost if and only if
\begin{equation}
 c_IU_J(M)\le f\bigl[P_J(M)-P_{J,I}(M)\bigr].
 \label{eq:marginal-break-even}
\end{equation}
SI Section~S8 applies both inequalities to a synthetic nested population.  The example is designed to separate two decisions: which of two invariants should be evaluated first, and whether adding a later invariant pays for itself.  There $I_2$ determines $I_1$, meaning that $I_1$ is a function of $I_2$, and $I_3$ determines $I_2$, meaning that $I_2$ is a function of $I_3$.  Thus each later invariant has finer fibers.  The fibers of $J_2$ are consequently the fibers of $I_2$, and the fibers of $J_3$ are the fibers of $I_3$; the same formulas apply when $I_1$ or $I_2$ is omitted.

\section{Discussion}
The main distinction is between random-pair separation and database uniqueness.  For a discrete fingerprint with collision mass $\alpha$, these are governed by $\alpha$ and $M^2\alpha$, respectively.  The results above make this distinction precise for the Alexander polynomial and the determinant, and the collision laws extend it to every discrete invariant.

For the determinant, the source of random-pair separation is local.  In each of Chapman's rooted random-diagram models there is a good-shadow event $\mathcal G_n$, with exponentially small complement, on which the shadow contains linearly many opened-trefoil slots.  After the exterior signs are exposed, the internal signs remain conditionally independent.  A single family of separating spheres gives the same exterior connected-sum factor for every internal redecoration.  Each trefoil slot then contributes one factor of $3$ to the determinant.  Hence, on $\mathcal G_n$,
\[
 v_3(\det K(D_n))=v_3(\det K_n^{\mathrm{ext}})+B_n,
 \qquad
 B_n\mid\mathcal F_n\sim\operatorname{Bin}(\lfloor\delta n\rfloor,1/4).
\]
On every $\mathcal F_n$-fiber contained in $\mathcal G_n$, the largest conditional atom is $O(n^{-1/2})$, so the conditional collision entropy tends to infinity after the exterior crossing data have been fixed.  Adding the probability of $\mathcal G_n^c$ gives the unconditional determinant anti-concentration estimate.

The knot argument supplies an anti-concentration estimate for one discrete law; the passage from a single comparison to a database is purely probabilistic.  For two independent samples from a common discrete law, the pair-collision mass $\alpha$, equivalently $H_2=-\log\alpha$, determines the birthday scale $M_{\mathrm B}=\alpha^{-1/2}$.  For a finite population sampled without replacement, the elementary-symmetric formula gives the no-collision probability for every database size.  This probability depends on the full fiber profile.  At fixed $T$ and $R$, balanced fibers maximize it, while imbalance can move birthday percentiles earlier.  Recent finite-census studies likewise give different rankings by value count, singleton detection, and random-pair separation \cite{dlotkogurnarisazdanovic2025,tubbenhauerzhang2025,kelomakietal2025}.

Collision probability also controls computational fallback.  An invariant-first procedure gives a negative answer when the two invariant values differ and invokes a complete procedure when they agree.  Thus, for two independent diagrams with law $\mu_n$, $\alpha_n(I)$ is the probability that the comparison reaches the complete branch.  In a database of independent diagrams with the same law, the number of complete calls in an all-pairs comparison is the number of colliding pairs.  For two independent diagrams from either Chapman ensemble, determinant comparison gives a polynomial-time negative answer with probability $1-O(n^{-1/2})$.  The difficulty of colliding pairs and the cost of the complete procedure lie outside this conclusion.

For learned knot classifiers, collision fibers provide a controlled test.  A classifier may perform well because it has recovered a familiar invariant, or because it detects information that the invariant misses.  To distinguish these possibilities, it can be tested inside a common invariant fiber.  Refining the fingerprint gives nested sets
\[
 \mathcal B_n(\det)
 \supseteq \mathcal B_n(\Delta)
 \supseteq \mathcal B_n(\Delta,V)
 \supseteq\cdots,
\]
where $V$ denotes the Jones polynomial.  Let $D,D'$ be independent with law $\mu_n$.  Whenever the conditioning event has positive probability, one may report
\[
 \Pi_n(f;J)=
 \Pr\{f(D,D')=\textsc{No}\mid J(D)=J(D'),\ K(D)\not\cong K(D')\}.
\]
This is the conditional negative accuracy for inequivalent pairs on which $J$ agrees.  It should be reported together with accuracy on equivalent pairs, since the constant-\textsc{No} rule has $\Pi_n(f;J)=1$.  Values of $\Pi_n(f;J)$ for different fingerprints refer to different conditioned laws and need not be monotone.  Such pairs may be used as training or test examples for information not contained in $J$.  Recent models distinguish some pairs sharing familiar invariants; $\Pi_n(f;J)$ tests this ability directly \cite{sleimanetal2024,halversonruehle2025}.

Two quantitative problems remain: whether the bound $\alpha_n(\det)=O(n^{-1/2})$ is sharp, and whether several local coordinates can improve the exponent.  The first also marks the difference between random-pair and fixed-target estimates.  For two independent diagrams, determinant equality is bounded by the largest atom and has probability $O(n^{-1/2})$.  A fixed target bounds the hidden binomial variable by a constant, giving an exponentially small lower tail.  The estimate $\alpha_n(\det)=O(n^{-1/2})$ proves that an independent database with $M_n=o(n^{1/4})$ has no determinant collision with high probability.  No matching lower bound is known, so the true order of $\alpha_n(\det)$ and its associated database transition remain open.

An improved exponent would follow from a genuinely $r$-dimensional conditional local limit theorem.  Suppose that, conditional on an exterior sigma-field, an $r$-component valuation vector is the sum of $s_n\asymp n$ independent bounded lattice vectors, its conditional covariance is comparable to $nI_r$, and it satisfies a uniform aperiodicity condition.  Fourier inversion would then give maximal-atom decay of order $n^{-r/2}$.  A separate problem is to control the dependence between the exterior valuation vector and the local sum when computing the collision mass.

Neither rate question has a law-independent answer.  Chapman's ensembles are uniform on rooted diagrams of fixed crossing number; finite censuses are sampled without replacement from a fixed list; and a law on knot types weights the topology differently.  Other random-knot models include lattice and off-lattice self-avoiding polygons, Petaluma diagrams, and random two-bridge knots \cite{sumnerswhittington1988,pippenger1989,evenzoharetal2016,evenzoharetal2018,cohenetal2018,cantarellaetal2026}.  The collision formulas apply once the law is specified, but their values and asymptotic rates require separate analysis.  In another random-knot model, an argument of the present type would require a multiplicative invariant whose valuation contains conditionally independent local summands on an event of high probability.

Finite catalogs present a further decision: which invariant should be evaluated next on entries that the preceding invariants did not separate?  The cascade formulas give the expected number of entries evaluated at each stage, the unresolved-pair count, the preferred order of two invariants, and the break-even cost of another invariant, without assuming independence among the invariant coordinates.  Random-pair separation, uniqueness in a database, and usefulness within a cascade are distinct properties of an invariant.  The same invariant may behave quite differently in the three settings.

\section{Materials and methods}
\label{sec:methods}
\paragraph{Collision laws.}
For fixed $n$, the invariant is a discrete random variable.  Its atom probabilities are $(p_v)_v$, its pair-collision mass $\alpha_n=\sum_vp_v^2$ is the probability that two independent values agree, and its triple mass $\beta_n=\sum_vp_v^3$ is the probability that three agree.  The variable $C_M$ is the database collision count defined in Section~\ref{sec:birthday-law}.  Its expectation and variance follow by classifying two unordered pairs of sample indices according to whether they coincide, share one index, or are disjoint.  Markov's and Chebyshev's inequalities give the distribution-free zero--one law.  The critical-window Poisson statement follows from factorial moments under the additional condition $M_n^3\beta_n\to0$.  This condition controls clusters of collision pairs that share an observation.  The pair-collision condition alone does not give a universal Poisson limit: a distinguished atom whose mass moves with $n$ can preserve such clusters at the critical scale.  The calculations and the counterexample exhibiting this obstruction appear in \SILocation{\SIExactCollisionCalculus{}, \SIUniversalBirthdayLaw{}, \SIPoissonWindow{}, and \SIHeavyAtomCounterexample{}}.

\paragraph{Random diagrams and trefoil slots.}
The proof separates a probabilistic statement about the underlying shadow from a topological calculation in disjoint local balls.  We use the rooted general one-component and rooted reduced one-component diagram classes at fixed crossing number developed by Cantarella, Chapman, and Mastin and analyzed asymptotically by Chapman \cite{cantarellachapmanmastin2016,chapman2017}.  Forgetting the crossing information gives the corresponding rooted shadow.  Every $n$-crossing shadow in either class has exactly $2^n$ decorated lifts, so a uniform diagram is obtained from a uniform rooted shadow by assigning independent fair crossing signs.

Outside an event of exponentially small probability, Chapman's connected-sum pattern theorem gives linearly many copies of a fixed asymmetric decoration of the opened-trefoil shadow.  We discard the witness signs and apply a deterministic rule to the shadow, before any signs are exposed, to choose $s_n$ occurrences at pairwise graph distance at least three.  Each selected occurrence contains three crossing vertices and the five edges between them, and exactly two distinct joining edges connect it to the rest of the shadow.  The separator lemma gives a shadow-measurable family of pairwise disjoint balls that works for every simultaneous internal redecoration.

We expose all exterior signs before sampling the $3s_n$ internal signs.  Replacing the selected local strands by fixed trivial arcs then gives the same exterior knot $K_n^{\mathrm{ext}}$ for every internal sign vector.  The local two-braid calculation shows that each closure is a trefoil for two of the eight decorations and an unknot for the other six.  Conditional on the exposed shadow and exterior signs, the trefoil indicators are therefore independent Bernoulli$(1/4)$ variables.  These statements are proved in \SILocation{\SIShadowSignLaw{}, \SIShadowSelection{}, \SIGenuineInterface{}, \SIDistanceThreeBuffer{}, \SISimultaneousSeparators{}, \SITwoBraidClosure{}, \SIPersistentFactorization{}, and \SIDeterminantAntiConcentration{}}.

\paragraph{Determinant and computation.}
The determinant is normalized as $\det K=|\Delta_K(-1)|$.  On the good-shadow event, multiplicativity under connected sum gives the valuation identity in \eqref{eq:hidden-binomial}.  A Fourier bound for the maximal atom of $\operatorname{Bin}(m,1/4)$ proves the moving-target theorem (\SILocation{\SIBinomialAtom{} and \SIDeterminantAntiConcentration{}}), and a Chernoff-type lower-tail estimate gives the fixed-target exponential bound (\SILocation{\SIFixedDeterminant{} and \SIFixedKnot{}}).  For the algorithmic statements, the determinant is computed from a reduced Goeritz matrix.  Fraction-free elimination and a Hadamard bound give polynomial bit complexity in the diagram size \cite{gordonlitherland1978,bareiss1968}.  A complete knot-equivalence procedure makes $\mathsf A_{\det}$ defined for every pair \cite{haken1961,hemion1992}.  General complexity results for knot and link decision problems are given in \cite{hasslagariaspippenger1999,lackenby2021}.  No running-time bound for the complete branch is used here.

\paragraph{Finite populations and cascades.}
A finite census is modeled as a fixed population of $T$ distinct entries partitioned into invariant fibers.  Databases are uniform $M$-subsets, so all formulas use sampling without replacement.  No-collision probabilities are computed from the elementary symmetric polynomial $e_M(m_1,\ldots,m_R)$ by integer dynamic programming.  At fixed $T$ and $R$, the extremal results follow from pairwise balancing operations on the integer fiber profile.  Cascade calculations use explicit joint prefix fibers and count sampled singletons, unresolved objects, and unresolved pairs.  The examples in Sections~\ref{sec:finite-databases} and \ref{sec:cascades} are finite labeled populations specified in the SI, not empirical knot censuses.  The finite-population proofs are \SILocation{\SIFiniteMoments{}--\SISupportCompression{}}, the cascade formulas are \SILocation{\SICascadeSingletons{}--\SICascadeBreakEven{}}, and the accompanying data and programs are described in \SILocation{\SIReproducibilitySection{}}.

\paragraph{Use of generative AI.}
During the preparation of this manuscript, the authors used ChatGPT and Codex (OpenAI), Claude (Anthropic), and Gemini (Google) for exploratory discussion, language editing, code review, and bibliographic checking. The research questions, mathematical ideas, theorem statements, and proofs were developed and verified by the authors. No output from these systems was treated as a mathematical source. No generative-image tools were used; all figures were produced deterministically from data and \LaTeX/TikZ source. The authors take full responsibility for the article.

\section*{Author contributions}
Conceptualization: P.O.-S., R.S., and E.L.; formal analysis: P.O.-S., E.L., E.J.V.G., C.A.R.G., A.G.-S., R.O.S.-G., and R.L.-V.; writing---original draft: P.O.-S. and E.L.; writing---review and editing: all authors; supervision: E.L.

\section*{Competing interests}
The authors declare no competing interest.

\section*{Data, materials, and software availability}
All data, code, benchmark populations, and figure sources used in this study are included in the accompanying reproducibility archive. The archive will be deposited in a public repository upon publication and is available to editors and reviewers with the submission.

\section*{Acknowledgments}
This work grew in part from the undergraduate honors thesis ``Hubs, Bridges and the Complexity of Unknotting: A structural and computational study of Clasp Diagram simplification'' of Pedro Olivares-S\'anchez in the Applied Mathematics program at the Instituto Tecnol\'ogico Aut\'onomo de M\'exico, and from research undertaken by Ramiro L\'opez-V\'azquez for his PhD thesis in the Department of Mathematics at CINVESTAV. Both theses were directed by Ernesto Lupercio. Part of the research was carried out during Ernesto Lupercio's sabbatical at the International Center for Mathematical Sciences--Sofia, Institute of Mathematics and Informatics, Bulgarian Academy of Sciences. E.L. thanks ICMS--Sofia and IMI--BAS for their hospitality and excellent working conditions. R.S. was partially supported by a Simons Travel Award. E.L. acknowledges financial support from the Consejo Nacional de Humanidades, Ciencias y Tecnolog\'ias of Mexico, Grant CB-2017-2018-A1-S-30345. The ICMS--Sofia portion of this research was supported by the Simons Foundation under Grant SFI-MPS-T-Institutes-00007697 and by the Ministry of Education and Science of the Republic of Bulgaria under contract No.~DO1-239/10.12.2024.

\end{document}